\documentclass{article}

\usepackage{arxiv}

\usepackage[utf8]{inputenc} 
\usepackage[T1]{fontenc}    
\usepackage{hyperref}       
\usepackage{url}            
\usepackage{booktabs}       
\usepackage{amsfonts}       
\usepackage{nicefrac}       
\usepackage{microtype}      
\usepackage{lipsum}
\usepackage{graphicx}
\usepackage{lipsum}
\usepackage{epstopdf}
\usepackage{algorithmic}
\usepackage[numbers]{natbib}
\usepackage{bm}
\usepackage{mathtools}
\usepackage{enumitem}
\usepackage{amsthm}
\graphicspath{ {./images/} }

\newtheorem{theorem}{Theorem}
\newtheorem{lemma}{Lemma}
\newtheorem{remark}{Remark}

\newtheorem*{assumptionsg*}{Assumptions G}
\newtheorem*{assumptionsb*}{Assumptions B}
\newtheorem*{algorithmI*}{Algorithm I}

\usepackage{amsopn}
\DeclareMathOperator{\diag}{diag}

\title{Global Convergence of a Line-Search Filter Differential Dynamic Programming Method}

\author{
 Ming Xu \\
  School of Computer and Communication Sciences\\
  EPFL\\
  1015 Lasanne, Switzerland \\
  \texttt{mxu925@gmail.com} \\
  \And
  Iman Shames \\
  Department of Electrical and Electronic Engineering \\
  University of Melbourne\\
  Parkville, VIC, 3010, Australia\\
  \texttt{iman.shames@unimelb.edu.au} \\
}

\begin{document}
\maketitle
\begin{abstract}
In this article, we establish the global convergence properties of the FilterDDP algorithm, which extends the discrete-time differential dynamic programming (DDP) algorithm of Mayne and Jacobson [\emph{International Journal of Control}, 3, (1966), pp. 85-95] to handle nonlinear constraints over states and controls, in addition to the dynamics.  FilterDDP adopts a line-search filter procedure for step acceptance. However, instead of a damped Newton step applied in the general nonlinear programming setting, the computation of a trial point involves applying a backward recursion and a forward simulation. We establish the global convergence of FilterDDP by showing that for a subset of constrained optimal control problems, the this backward-forward procedure satisfies the same properties as a Newton step for the purpose of establishing global convergence of a line-search filter method, following the analysis of Wächter and Biegler [\emph{SIAM Journal on Optimization}, 16 (2005), pp. 1-31].
\end{abstract}


\section{Introduction}

The differential dynamic programming (DDP) algorithm, originally proposed by Mayne \cite{mayneddp} is a structure exploiting algorithm for solving unconstrained, discrete-time optimal control problems (OCPs). Important properties of DDP include: 1) each iterate of the algorithm satisfies the discrete-time dynamics equations, 2) each iteration is of linear time and memory with respect to horizon length and, 3) a time-varying, affine state-feedback policy is provided at each iteration instead of a control sequence. The local quadratic convergence of the DDP algorithm is established formally in Liao and Shoemaker~\cite{liaoddplocal} and Murray and Yakowitz~\cite{murrayddplocal}. Global convergence of DDP with an inexact line search was established by Yakowitz and Rutherford~\cite{yakowitzddpglobal}. Recently, Roulet et al.~\cite{rouletjmlrddp} provide an alternative proof of the local quadratic and global convergence of DDP, without a line-search procedure, for a restricted class of OCPs.

There are several extensions to the unconstrained DDP of Mayne \cite{mayneddp} for solving constrained optimal control problems, e.g.,~\cite{murraycddpwater, lantoinejotahddp, tassaclddp, altro, proxddp, aoyama2024secondorderconstraineddynamicoptimization, pavlov_ipddp, filterddp}. However, only a subset of the existing constrained DDP algorithms proposed in the literature are accompanied by a formal convergence analysis. For example, Boutselis et al.~\cite{ddpliegroup} extend DDP to the setting where states are constrained to evolve on a Lie group and establish global convergence by extending the result for unconstrained DDP from~\cite{yakowitzddpglobal}. Yakowitz \cite{yakowitzcddp} propose a DDP algorithm which handles nonlinear inequality constraints on states and controls, based on a stage-wise sequential quadratic programming formulation. Global convergence of the algorithm is established in a simplified setting, assuming a convex objective function and linear constraints. Ohno~\cite{ohnoconstrainedddp} propose a DDP algorithm which handles nonlinear equality and inequality constraints on states and controls, based on a primal-dual formulation. The local superlinear convergence of the algorithm is established, although no discussion of global convergence is provided. Similarly, Pavlov et al.~\cite{pavlov_ipddp} propose a primal-dual interior point DDP algorithm and establish the local quadratic convergence of the algorithm. Again, no global convergence analysis is provided. 

Recently, Xu et al.~\cite{filterddp} proposed a line-search filter DDP algorithm based on the algorithm proposed by W\"{a}chter and Biegler \cite{wachterglobal} for the general, nonlinear programming (NLP) case, with the damped Newton step for trial point determination replaced with a DDP style iteration. The local quadratic convergence of the algorithm was established \cite{filterddp} after ignoring the step acceptance criteria and a numerical implementation was provided. However, the global convergence analysis of FilterDDP was not established. The purpose of this article is to formally establish the global convergence of the FilterDDP algorithm, by appropriately adapting the global convergence analysis in W\"{a}chter and Biegler~\cite{wachterglobal}.

The paper is organized as follows. We first present the FilterDDP algorithm and global convergence result for constrained optimal control problems without inequality constraints for ease of comprehension. Subsequently, we will describe an extension to the barrier interior point framework for handling inequality constraints. We present our results in a similar style and structure to W\"{a}chter and Biegler~\cite{wachterglobal} for clarity.

In section \ref{sec:coc}, we state the constrained, discrete-time optimal control problems under consideration in this paper. In section \ref{sec:ddp}, we describe the FilterDDP algorithm, including the backward recursion and forward simulation phases for computing a trial point, as well as the adapted line-search filter step acceptance criteria from \cite{wachterglobal}. Subsequently, we establish the global convergence of the FilterDDP algorithm in section \ref{sec:global}, using assumptions around regularity and smoothness similar to those of \cite{wachterglobal} in the general NLP setting. Specifically, we establish that every limit point of the sequence of iterates generated by the FilterDDP algorithm is feasible, and that there is at least one limit point that satisfies the first order optimality conditions for the problem. Finally, section \ref{sec:ipmextension} presents and establishes the global convergence of a barrier interior point extension of FilterDDP.

\subsection{Notation}

We denote the $i$th component of a vector $v\in\mathbb{R}^n$ by $v^{(i)}$. Let $\|\cdot\|$ return the norm of its argument. Let $[N]$ represent the set of integers from $1$ to $N$, i.e., $\{1, \dots, N\}$. A vector of ones of appropriate size is denoted by $e$. Denote the set of indexed vectors $\{y_k\}_{k=t}^{N}$ by $y_{t:N}$ and for brevity, we use the convention $(x, u) = (x^\top, u^\top)^\top$. We use $\odot$ to denote the element-wise product (Hadamard product). For a matrix $A$, we denote by $\sigma_{\min}(A)$ the smallest singular value of $A$, and for a symmetric, positive definite matrix $A$, we call the smallest eigenvalue $\lambda_{\min}(A)$. Denote by $O(s_{k,t})$ a sequence $\{v_{k,t}\}$ satisfying $\|v_{k, t}\| \leq G s_{k, t}$ for some constant $G > 0$ independent of $k, t$. Finally, we use the following convention for derivatives: If $f$ is a scalar valued function $f : \mathbb{R}^n \times \mathbb{R}^m \rightarrow \mathbb{R}$, then $\nabla_u f(x, u) \in \mathbb{R}^{1\times m}$, i.e., a row vector and $\nabla_{xu}^2 f(x) \in \mathbb{R}^{n\times m}$. Furthermore, if $g:\mathbb{R}^n\times \mathbb{R}^m \rightarrow \mathbb{R}^c$ is a vector valued function, then $\nabla_x g(x, u) \in \mathbb{R}^{c\times n}$ and $\nabla_{xu}^2 g(x, u) \in\mathbb{R}^{c \times n \times m}$. Given a vector $\lambda\in\mathbb{R}^c$ and tensor $A \in \mathbb{R}^{c\times n \times m}$, denote a tensor contraction along the first dimension of $A$ by $\lambda \cdot A  \in \mathbb{R}^{n \times m}$.

\section{Constrained Optimal Control}\label{sec:coc}

For simplicity, we begin by describing the FilterDDP algorithm for discrete-time, finite horizon optimal control problems with only nonlinear equality constraints, given by
\begin{equation}\label{eq:coc}
\begin{array}{rl}
    \underset{x_{1:N}, u_{1:N}}{\text{minimize}}  & \sum_{t=1}^{N} \ell(x_t, u_t) \\
    \text{subject to} & x_1 = \hat{x}_1, \\
    & x_{t+1} = f(x_t, u_t) \quad \text{for } t \in [N-1],\\
    & c(x_t, u_t) = 0  \quad \text{for } t \in [N],
\end{array}
\end{equation} 
where $\mathbf{x}\coloneqq x_{1:N}$ and $\mathbf{u} \coloneqq u_{1:N}$ are a trajectory of states and control inputs,
respectively and $x_t\in\mathbb{R}^{n_x}$, $u_t\in\mathbb{R}^{n_u}$ for all $t\in[N]$. The stage objective functions are denoted by
$\ell:\mathbb{R}^{n_x} \times\mathbb{R}^{n_u} \rightarrow \mathbb{R}$ and the
 constraints are represented by mappings $c:\mathbb{R}^{n_x\times n_u} \rightarrow \mathbb{R}^{n_c}$
where $n_c \leq n_u$. The mappings $f:\mathbb{R}^{n_x} \times\mathbb{R}^{n_u} \rightarrow \mathbb{R}^{n_x}$ captures the dynamics of the system in discrete time and
$\hat{x}_1$ is the (known) initial state. We require that $\ell, f, c$ are twice continuously differentiable.

We describe an extension to further include inequality constraints and provide its associated global convergence analysis in Section~\ref{sec:ipmextension}.

\begin{remark}
    The objective, dynamics and constraint functions can in general be time-varying but we avoid specifying this explicitly for notational simplicity.
\end{remark}

\subsection{Optimality Conditions}\label{ssec:kkteq}
The Lagrangian of~\eqref{eq:coc} is given by
\begin{equation}\label{eq:constrlagrangian}
    \mathcal{L}(\mathbf{x}, \mathbf{u}, \bm{\phi}, \bm{\lambda}) \coloneqq \lambda_1^\top \left( \hat{x}_1 - x_1 \right) + \sum_{t=1}^{N} L(x_t, u_t, \phi_t) 
    + \sum_{t=1}^{N-1} \lambda_{t+1}^\top \bigl( f(x_t, u_t) - x_{t+1}\bigr),
\end{equation}
where $\bm{\phi} \coloneqq \phi_{1:N}$ and $\bm{\lambda} \coloneqq \lambda_{1:N}$ are Lagrange multipliers associated with the  equality constraints $c$ and the initial state and dynamics constraints $f$, respectively and $L(x, u, \phi) \coloneqq \ell(x, u) + \phi^\top c(x, u)$. 
The Karush-Kuhn-Tucker (KKT) conditions for OCP \eqref{eq:coc} state that for $\mathbf{x}, \mathbf{u}$ to be a (local) solution of \eqref{eq:coc}, there must exist Lagrange multipliers $\bm{\phi}$ and $\bm{\lambda}$ such that for all $t\in[N]$,
  \begin{subequations}\label{eq:kkteq}
    \begin{align}
        \nabla_{x_t}\mathcal{L}(\mathbf{x}, \mathbf{u}, \bm{\phi}, \bm{\lambda}) =& \, \nabla_{x_t} \left(L(x_t, u_t, \phi_t) + \lambda_{t+1}^\top f(x_t, u_t)\right) - \lambda_t^\top  = 0, \label{eq:constrkktx}\\
        \nabla_{u_t}\mathcal{L}(\mathbf{x}, \mathbf{u}, \bm{\phi}, \bm{\lambda}) =& \, \nabla_{u_t} L(x_t, u_t, \phi) + \lambda_{t+1}^\top \nabla_{u_t}f(x_t, u_t)  = 0, \label{eq:constrkktu}\\
      \nabla_{\lambda_t} \mathcal{L}(\mathbf{x}, \mathbf{u}, \bm{\phi}, \bm{\lambda}) =&  \, \begin{cases}
          (\hat{x}_1 - x_t)^\top = 0 & t = 1 \\
          (f(x_{t-1}, u_{t-1}) - x_{t})^\top  =  0 & t > 1
      \end{cases},\label{eq:constrkktdyn}\\
      \nabla_{\phi_t}\mathcal{L}(\mathbf{x}, \mathbf{u}, \bm{\phi}, \bm{\lambda}) =& \, c(x_t, u_t)^\top =  0,\label{eq:constrkktconstr}
    \end{align}
  \end{subequations}
noting that we implicitly set $\lambda_{N+1} = 0$ since there are no dynamics constraints for step $t=N$. We define a \emph{KKT point} to be an iterate $\mathbf{x}_k$, $\mathbf{u}_k$ for which there exist $\bm{\lambda}_k$ and $\bm{\phi}_k$ satisfying the KKT conditions~\eqref{eq:kkteq}. Under certain constraint qualifications (e.g., linear independence of the constraint gradients), the KKT conditions are the first-order optimality conditions for \eqref{eq:coc}.

\paragraph*{Dynamic feasibility} A core design principle of DDP-style algorithms is that the dynamics constraints \eqref{eq:constrkktdyn} are automatically satisfied for each iteration due to the forward simulation phase described in Section \ref{ssec:derivforward} and the assumption of a known initial state in the OCP definition in \eqref{eq:coc}.

\section{FilterDDP}\label{sec:ddp}

The FilterDDP algorithm generates an sequence of iterates $\{(\mathbf{w}_k, \bm{\lambda})\}_{k\geq 0}$ to OCP~\eqref{eq:coc}, where $\mathbf{w}_k \coloneqq (\mathbf{x}_k, \mathbf{u}_k, \bm{\phi}_k)$, given an initial estimate $\mathbf{w}_{0}$. The algorithm design follows the line search filter method proposed in~\cite{wachterglobal}, with the damped Newton step applied to KKT conditions~\eqref{eq:kkteq} replaced with a backward recursion and forward simulation phase, to be described in this section. Local quadratic convergence of this algorithm around a critical point was established by Xu et al. \cite{filterddp}.

\subsection{Preliminaries} For a function $z \in \{\ell, f, c\}$ and $y_1, y_2 \in \{x, u\}$, let $z^{k, t}_{y_1}$ and $z^{k, t}_{y_1 y_2}$ be shorthand for the partial derivative of $z^t$ with respect to $y_1$, and $y_2$ followed by $y_1$, respectively, evaluated at $x_{k, t}, u_{k, t}$. For example, in the scalar valued function case, $\ell_{u}^{k, t} \coloneqq \nabla_{u} \ell(x_{k, t}, u_{k, t})$ and $\ell_{xu}^{k, t} \coloneqq  \nabla_{xu}^2 \ell(x_{k, t}, u_{k, t})$. Examples for the vector valued function case include $c_{x}^{k, t} \coloneqq \nabla_x c(x_{k, t}, u_{k, t})$ and $f_{xu}^{k, t} \coloneqq \nabla_{xu}^2 f(x_{k, t}, u_{k, t})$. Let $z^{k, t}$ be shorthand for $z$ evaluated at $x_{k, t}$, $u_{k, t}$, e.g., $c^{k, t} \coloneqq c(x_{k, t}, u_{k,t})$. 

\subsection{Backwards Pass}\label{ssec:recursivenewton}
The backward pass recursively updates the variables $r_{k, t}, \lambda_{k, t} \in \mathbb{R}^{n_x}$ and $P_{k, t} \in \mathbb{R}^{n_x\times n_x}$ symmetric for all $t\in[N+1]$ and $k$ with boundary conditions $r_{k, N+1}, \lambda_{k, N+1} \coloneqq 0$ and $P_{k, N+1} = 0$. We define the recursion as follows. Suppose for the current index $t\in[N]$ that $r_{k, t+1}, \lambda_{k, t+1}$ and $P_{k, t+1}$ are known. Let
\begin{equation}\label{eq:Q}
  Q^{k, t}(x, u, \phi) \coloneqq L(x, u, \phi) + V^{k, t+1}(f(x, u)),
\end{equation}
where
\begin{equation}
    V^{k, t}(x)\coloneqq r_{k, t}^\top x + \frac{1}{2} x^\top P_{k, t} x.
\end{equation}

Noting that $V^{k, t}_x \coloneqq \nabla_x V^{k, t}(x_{k, t}) = r_{k, t}^\top + x_{k, t}^\top P_{k, t}$ and $V_{xx}^{k, t} \coloneqq \nabla^2_{xx} V^{k, t}(x_{k, t}) = P_{k, t}$ and differentiating $Q^{k, t}$ at the nominal trajectory $x_{k, t}, u_{k, t}$ and $\phi_{k, t}$ yields
\begin{subequations}\label{eq:Qderivs}
\begin{gather}
  Q_{uu}^{k, t} \coloneqq L_{uu}^{k, t} + (f_u^{k, t})^\top P_{k, t+1} f_u^{k, t} + V_{x}^{k, t+1} \cdot f_{uu}^{k, t} , \label{eq:Quut}\\
  Q_{ux}^{k, t} \coloneqq L_{ux}^{k, t} + (f_u^{k, t})^\top P_{k, t+1} f_x^{k, t} + V_{x}^{k, t+1} \cdot f_{ux}^{k, t}, \label{eq:Quxt} \\
  Q_{xx}^{k, t} \coloneqq L_{xx}^{k, t} + (f_x^{k, t})^\top P_{k, t+1} f_x^{k, t} + V_{x}^{k, t+1} \cdot f_{xx}^{k, t}, \label{eq:Qxxt} \\
  Q_u^{k, t} \coloneqq \ell_u^{k, t} +  V_{x}^{k, t+1} f_u^{k, t} + \phi_{k, t}^\top c_u^{
  k, t}, \label{eq:gt}\\
  Q_x^{k, t} \coloneqq \ell_x^{k, t} + V_{x}^{k, t+1} f_x^{k, t} + \phi_{k, t}^\top c_x^{k, t}. \label{eq:Qx}
\end{gather}
\end{subequations}

Next, define the perturbed hessians introduced by \cite{filterddp}, given by
\begin{subequations}\label{eq:hessiansperturbedexact}
    \begin{gather}
    \tilde{H}_{k, t} \coloneqq L_{uu}^{k, t} + (f_u^{k, t})^\top P_{k, t+1} f_u^{k, t} + \lambda_{k, t+1} \cdot f_{uu}^{k, t},  \label{eq:Ht}\\
    \tilde{B}_{k, t} \coloneqq L_{ux}^{k, t} + (f_u^{k, t})^\top P_{k, t+1} f_x^{k, t} + \lambda_{k, t+1} \cdot f_{ux}^{k, t} , \label{eq:Bt} \\
    \tilde{C}_{k, t} \coloneqq L_{xx}^{k, t} + (f_x^{k, t})^\top P_{k, t+1} f_x^{k, t} + \lambda_{k, t+1} \cdot f_{xx}^{k, t}. \label{eq:Ct}
  \end{gather}
\end{subequations}

Local quadratic convergence of FilterDDP using \eqref{eq:hessiansperturbedexact} to determine the update rule in the backward pass was established in \cite{filterddp}. Next, we introduce update rule parameters $\zeta_{k, t}, \psi_{k, t}$, $\beta_{k, t}, \omega_{k, t}$, which satisfy the relation
\begin{equation}\label{eq:bwlinear}
    \begin{bmatrix}
        H_{k, t}  & A_{k, t} \\
        A_{k, t}^\top & 0 
    \end{bmatrix}
    \begin{bmatrix}
        \zeta_{k, t} & \beta_{k, t} \\ \psi_{k, t} & \omega_{k, t}
    \end{bmatrix}
     = -\begin{bmatrix}
         (Q_u^{k, t})^\top & B_{k, t} \\ c^{k, t} & c_x^{k, t}
     \end{bmatrix},
\end{equation}
where $H_{k, t}$, $B_{k, t}$ and $C_{k, t}$ are bounded approximations to $\tilde{H}_{k, t}$, $\tilde{B}_{k, t}$ and $\tilde{C}_{k, t}$. Moreover, to establish the global convergence result in section \ref{sec:global}, it is required that $H_{k,t}$ be uniformly positive definite in the null space of the constraint Jacobian $A_{k,t}^\top$ for all $t\in[N]$ and $k$. Finally, we complete the recursion step by defining
\begin{subequations}\label{eq:vderivsupdate}
\begin{gather}
 \lambda_{k, t} \coloneqq (L_{x}^{k, t} + \phi_{k, t}^\top c_{k, t} + \lambda_{k, t+1}^\top f_{x}^{k, t})^\top \\
    P_{k, t} \coloneqq C_{k, t} + \beta_{k, t}^\top H_{k, t} \beta_{k, t} + B_{k, t}^\top \beta_{k, t} + \beta_{k, t}^\top B_{k, t},\label{eq:vxxupdate} \\
  r_{k, t} \coloneqq Q_x^{k, t} + Q_u^{k, t} \beta_{k, t} + (c^{k, t})^\top \omega_{k, t} - P_{k, t} x_{k, t}. \label{eq:vxupdate}
\end{gather}
\end{subequations}
The recursion begins by setting $t=N$ and proceeds backwards in time until $t=0$.

\subsection{Forward Simulation}\label{ssec:derivforward}

After the backward phase in Section~\ref{ssec:recursivenewton} is complete, a forward simulation phase can be applied to determine the updated iterate for iteration $k+1$. In particular, we apply a nonlinear update rule using forward simulation given by
\begin{gather}
    u_{k+1, t} = u_{k, t} + \alpha_k \zeta_{k, t} + \beta_{k, t} (x_{k+1, t} - x_{k, t}), \label{eq:forwardu} \\
    \phi_{k+1, t} = \phi_{k,t} + \alpha_k \psi_{k, t} + \omega_{k, t} (x_{k+1, t}- x_{k, t}), \label{eq:forwardphi}\\
    x_{k+1, t+1} = f(x_{k+1, t}, u_{k+1, t}),
\end{gather}
where $\alpha_k \in (0, 1]$ is the step size for iteration $k$ selected independent of $t$, determined using a backtracking line search where a decreasing sequence of step sizes $\alpha_{k}^l \in (0, 1]$, $l=\{0, 1, 2, \dots\}$ is tried until certain step acceptance criteria are satisfied. After determining the step size and thus, the next iterate, the algorithm proceeds to the next iteration. The criteria for accepting a trial point with step size $\alpha_{k}^l$ is based on the filter line search procedure described in~\cite{wachterglobal}, using the Lagrangian (as opposed to the objective) as one of the filter criteria. Global convergence of the method under the Lagrangian was established in section 4.1 of~\cite{wachterglobal}. 

The two filter criteria, namely the Lagrangian and constraint violation measure of OCP~\eqref{eq:coc}, are given by
\begin{subequations}\label{eq:filtermeasures}
    \begin{gather}
      \mathcal{L}(\mathbf{w}) \coloneqq \sum_{t=1}^N \ell(x_t, u_t) + \phi_t^\top c(x_t, u_t) \quad \text{and} \label{eq:filterlagrange}\\
      \theta(\mathbf{w}) \coloneqq \sum_{t=1}^N \|c(x_t, u_t)\| \label{eq:filterconstr},
    \end{gather}
\end{subequations}
respectively.
Before defining the step acceptance criteria, we first introduce notation for a trial point $\mathbf{w}_k^+(\alpha)$ with arbitrary step size $\alpha\in[0, 1]$, iteratively, to satisfy
\begin{subequations}\label{eq:forwardls}
    \begin{align}
      u_{k, t}^+(\alpha) & \coloneqq u_{k, t} + \alpha \zeta_{k, t} + \beta_{k, t}(x^+_{k, t}(\alpha) - x_{k, t}), \label{eq:uforwardls} \\
      \phi_{k, t}^+(\alpha) &\coloneqq \phi_{k, t} + \alpha \psi_{k, t} + \omega_{k, t} (x^+_{k, t}(\alpha) - x_{k, t}), \label{eq:phiforwardls}\\
      x^+_{k, t+1}(\alpha) &\coloneqq f(x^+_{k, t}(\alpha), u^+_{k, t}(\alpha)), \quad x^+_{k, 1}(\alpha) \coloneqq \hat{x}_1, \label{eq:dynforwardls}
    \end{align}
\end{subequations}
for all $t\in[N]$. 
\begin{remark}
  The forward simulation phase of DDP of \eqref{eq:forwardls} for determining each iterate ensures that constraint~\eqref{eq:constrkktdyn} is always satisfied. Hence, $\mathcal{L}$ in \eqref{eq:filterlagrange} and $\mathcal{L}$ in \eqref{eq:constrlagrangian} are equivalent for all trial points $\mathbf{w}_k(\alpha)$ of FilterDDP for any $\bm{\lambda}$ and $\alpha\in(0, 1]$. For the remainder of the article, the specific $\mathcal{L}$ used will be indicated explicitly by its arguments.
\end{remark}

The step acceptance criteria adopted in this paper is similar to the one described in~\cite{wachterglobal}, and involves several criterion described in the following sections.

\subsection{Sufficient reduction} A trial step size $\alpha_{k}^l$ must provide a \emph{sufficient reduction} in either the Lagrangian or constraint violation, i.e., for constants $\gamma_\mathcal{L}, \gamma_\theta > 0$, either of
\begin{subequations}\label{eq:filtersufficient}
  \begin{gather}
      \mathcal{L}(\mathbf{w}^+_{k}(\alpha_{k}^l)) \leq (1 - \gamma_\mathcal{L}) \mathcal{L}(\mathbf{w}_{k}), \label{eq:filtersufficientJ}\\
      \theta(\mathbf{w}^+_{k}(\alpha_{k}^l)) \leq (1 - \gamma_\theta) \theta(\mathbf{w}_{k}), \label{eq:filtersufficienttheta}
  \end{gather} 
\end{subequations}
holds. To prevent iterates converging to a feasible but non-optimal point,~\eqref{eq:filtersufficient} is replaced with a different criteria when the \emph{switching condition} given by
\begin{equation}\label{eq:switching}
    \begin{gathered}
        m_k(\alpha_{k}^l) \coloneqq \alpha_{k}^l \left(\sum_{t=1}^N   Q_u^{k, t}  \zeta_{k, t}  + \psi_{k, t}^\top c^{k, t} \right) < 0 \quad \text{and} \\
        (-m_k(\alpha_{k}^l))^{s_L}(\alpha_{k}^l)^{1 - s_L} > \delta \theta(\mathbf{w}^+_{k}(\alpha_{k}^l))^{s_\theta}
    \end{gathered}
\end{equation}
holds, with fixed constants $\delta > 0$, $s_\theta > 1$, $s_L \geq 1$, and $m_k(\alpha_{k}^l)$ is a model of the change in objective.

If the switching condition in \eqref{eq:switching} holds, then we replace the sufficient decrease condition in \eqref{eq:filtersufficient} with an Armijo-type condition for the objective function given by
\begin{equation}\label{eq:filterarmijo}
    \mathcal{L}(\mathbf{w}^+_{k}(\alpha_{k}^l)) \leq \mathcal{L}(\mathbf{w}_{k}) +  \eta_\mathcal{L} m_k(\alpha_{k}^l).
\end{equation}
Following~\cite{wachterglobal}, iterations in which the updated iterates satisfy \eqref{eq:switching} and \eqref{eq:filterarmijo} are called ``$\mathcal{L}$-type" iterations.

\subsection{Filter as a taboo region} A filter, which defines a taboo region of the iterates in the $\{(\theta, \mathcal{L}) \in \mathbb{R}^2: \theta \geq 0\}$ half-plane, is maintained to prevent the algorithm from \emph{cycling}, e.g., alternating between two points that decrease one of the measures $\theta$ and $\mathcal{L}$ while increasing the other. 

Following~\cite{wachterglobal}, we denote the filter by a set $\mathcal{F}_k\subseteq [0, \infty) \times \mathbb{R}$ containing all $(\theta, \mathcal{L})$ pairs that are prohibited in iteration $k$. A point is considered acceptable to the filter if
\begin{equation}\label{eq:filteracceptable}
    \left(\theta(\mathbf{w}^+_{k}(\alpha_{k}^l)), \, \mathcal{L}(\mathbf{w}^+_{k}(\alpha_{k}^l))\right) \notin \mathcal{F}_k.
\end{equation}

At the beginning of the optimisation, we initialise $\mathcal{F}_0 = \emptyset$\footnote{Alternatively, we can initialise the filter with an upper bound on the constraint violation, i.e., $\mathcal{F}_0 = \{(\theta, L) \in \mathbb{R}^2 : \theta \geq \theta_{\max}\}$ for $\theta_{\max}> 0$, without affecting the global convergence result, see~\cite{wachterglobal}.} and augment the filter in some iterations after the new iterate $\mathbf{w}_{k+1}$ is accepted, using the formula
\begin{equation}\label{eq:filterupdate}
      \mathcal{F}_{k+1} = \mathcal{F}_k \cup \left\{(\theta, \mathcal{L}) \in \mathbb{R}^2: \theta \geq (1 - \gamma_\theta) \theta(\mathbf{w}_{k}) \quad
      \text{and} \quad \mathcal{L} \geq \mathcal{L}(\mathbf{w}_{k}) - \gamma_\mathcal{L}  \theta(\mathbf{w}_{k})  \right\}. 
\end{equation}
If the filter is not augmented in the current iteration, then it is unchanged, i.e., $\mathcal{F}_{k+1} = \mathcal{F}_k$. Furthermore, the filter is only augmented in iterations which are not $\mathcal{L}$-type iterations, i.e., when~\eqref{eq:switching} does not hold for the accepted step size $\alpha_{k}$. 

\subsection{Feasibility restoration phase} In Lemma~\ref{lem:constrlinear}, we show that the constraint violation can be reduced at every iteration, i.e., $\theta(\mathbf{w}^+_{k}(\alpha_{k}^l)) < \theta(\mathbf{w}_{k})$ for a sufficiently small step size $\alpha_{k}^l$. However, a \emph{sufficient} decrease according to~\eqref{eq:filtersufficienttheta} is not guaranteed. As a result, the algorithm assumes the availability of a \emph{restoration phase}, whose purpose is to find a new iterate which satisfies~\eqref{eq:filtersufficient} and is acceptable to the filter, by trying to decrease the constraint violation. 

Following~\cite{wachterglobal}, we do not describe a specific method for this restoration phase. However, a concrete implementation for the restoration phase in the general NLP case is described in~\cite{ipopt}. In particular, an iterative method is proposed, which decreases the constraint violation while attempting not to deviate too far from the current iterate. The method in~\cite{ipopt} can be in principle, adapted to the FilterDDP algorithm.

As in~\cite{wachterglobal}, the FilterDDP algorithm switches to the restoration phase when the step size $\alpha_{k}^l$ is below a threshold $\alpha_k^{\min}$, where
\begin{equation}\label{eq:feasminstep}
    \alpha_k^{\min} \coloneqq \gamma_\alpha \cdot \begin{cases}
        \min \left\{ \gamma_\theta, \frac{\gamma_\mathcal{L} \theta(\mathbf{w}_k)}{-m_k(1)}, \frac{\delta\theta(\mathbf{w}_k)^{s_\theta}}{(-m_k(1))^{s_\mathcal{L}}} \right\} & \text{if} \, m_k(1) < 0,\\
        \gamma_\theta & \text{otherwise}
    \end{cases}.
\end{equation}
The conditions~\eqref{eq:feasminstep} are derived using linear models of the sufficient decrease conditions~\eqref{eq:filtersufficientJ},~\eqref{eq:filtersufficienttheta} and~\eqref{eq:switching} and a detailed derivation of~\eqref{eq:feasminstep} is presented in~\cite{wachterglobal}.

The algorithm also switches to the restoration phase if the constraint Jacobians $A_{k, t}^\top$ is (almost) rank deficient for any $t\in[N]$, i.e., $\sigma_{\min}(A_{k, t})$ is arbitrarily close to zero. We assume the algorithm is able to detect these cases.

\begin{remark}
    We defer the reader to~\cite{wachterglobal} for a detailed discussion around the motivation behind the aforementioned criteria for establishing global convergence the line search filter algorithm. These considerations equally apply to FilterDDP. 
\end{remark}

\subsection{The algorithm}
We now formally state the overall FilterDDP algorithm for solving constrained optimal control problems of form \eqref{eq:coc}.

\medskip

\begin{algorithmI*} \hfill \medskip

\textit{Given:} Starting point $\mathbf{w}_{0}$, constants $\theta_{\max} \in (\theta(\mathbf{w}_{0}), \infty]$; $\gamma_\theta, \gamma_\mathcal{L} \in (0, 1)$; $\delta > 0$; $\gamma_\alpha \in (0, 1]$; $s_\theta > 1$; $s_\mathcal{L} \geq 1$; $\eta_\mathcal{L} \in (0, \frac{1}{2})$; $0 < \tau_1 \leq \tau_2 < 1$.
\begin{enumerate}[label*=\arabic*.]
    \item \textit{Initialise.} \,\, Initialise the the iteration counter $k \gets 0$ and filter $\mathcal{F}_0 \coloneqq \left\{ (\theta, \mathcal{L}) \in \mathbb{R}^2 : \theta \geq \theta_{\max} \right\}$.
    \item \textit{Check convergence.} \,\, Stop if $\mathbf{w}_{k}$ satisfies KKT conditions~\eqref{eq:kkteq}. \label{alg:checkconverge}
    \item \textit{Backward pass.} \,\, Set $r_{k, N+1}, \lambda_{k, N+1} \gets 0$, $P_{k, N+1} \gets 0$ and time counter $t \gets N$.
    \begin{enumerate}[label*=\arabic*.]
        \item If $t = 0$, go to step~\ref{alg:commencels} \label{alg:beginbw}
        \item Solve for $\zeta_{k, t}, \beta_{k, t}, \psi_{k, t}$ and $\omega_{k, t}$ using~\eqref{eq:bwlinear}. If the matrix in~\eqref{eq:bwlinear} is too ill conditioned, go to the feasibility restoration phase in step~\ref{alg:restophase} \label{alg:restofrombw}
        \item Update $r_{k, t}$, $\lambda_{k, t}$, $P_{k, t}$ using~\eqref{eq:vderivsupdate}.
        \item Set $t \gets t - 1$ and go to~\ref{alg:beginbw}
    \end{enumerate}
    \item \textit{Backtracking line search.} \,\, 
        \begin{enumerate}[label*=\arabic*.]
            \item \textit{Initialise line search.} \,\, Set $\alpha_{k, 0} = 1$ and $l \gets 0$. \label{alg:commencels}
            \item \textit{Compute new trial point.} \,\, If the trial step size becomes too small, i.e., $\alpha_{k}^l < \alpha_k^{\min}$ where $\alpha_k^{\min}$ is defined in~\eqref{eq:feasminstep}, go to the feasibility restoration phase in step~\ref{alg:restophase} Otherwise, compute the new trial point $\mathbf{w}^+_{k}(\alpha_{k}^l)$ using the forward pass, i.e.,~\eqref{eq:forwardls}.  \label{alg:propnewtrialpoint}
            \item \textit{Check acceptability to the filter.} \,\, If~\eqref{eq:filteracceptable} does not hold, reject the trial step size and go to step~\ref{alg:newstepsize}
            \item \textit{Check sufficient decrease with respect to the current iterate.} \,\,
                \begin{enumerate}[label*=\arabic*.]
                    \item \textit{Case} I: $\alpha_{k}^l$ is a $\mathcal{L}$-step size (i.e.,~\eqref{eq:switching} holds). If the condition~\eqref{eq:filterarmijo} holds, accept the trial step and go to step~\ref{alg:accepttrial} Otherwise, go to step~\ref{alg:newstepsize}
                    \item \textit{Case} II: $\alpha_{k}^l$ is not a $\mathcal{L}$-step size, (i.e.,~\eqref{eq:switching} is not satisfied): If~\eqref{eq:filtersufficient} holds, accept the trial step and go to step~\ref{alg:accepttrial} Otherwise, go to step~\ref{alg:newstepsize}
                \end{enumerate}
        \end{enumerate}
        \item \textit{Choose new trial step size.} Choose $\alpha_{k}^{l+1} \in \left[\tau_1 \alpha_{k}^l, \tau_2 \alpha_{k}^l\right]$, set $l \gets l+1$ and go back to step \ref{alg:propnewtrialpoint} \label{alg:newstepsize}
        \item \textit{Accept trial point.} \, Set $\alpha_k \coloneqq \alpha_{k}^l$ and $\mathbf{w}_{k+1} \coloneqq \mathbf{w}_{k}^+(\alpha_{k})$. \label{alg:accepttrial}
        \item \textit{Augment filter if necessary.} \,\, If $k$ is not a $\mathcal{L}$-type iteration, augment the filter using \eqref{eq:filterupdate}; otherwise leave the filter unchanged, i.e., set $\mathcal{F}_{k+1} \coloneqq \mathcal{F}_k$. 
        \item \textit{Continue with next iteration.} \,\, Increase the iteration counter $k \gets k+1$ and go back to step \ref{alg:checkconverge} \label{alg:itercounterincre}
        \item \textit{Feasibility restoration phase.} Compute a new iterate $\mathbf{w}_{k+1}$ by decreasing  $\theta$ so that $\mathbf{w}_{k+1}$ satisfies the sufficient decrease conditions \eqref{eq:filtersufficient} and is acceptable to the filter, i.e., $(\theta(\mathbf{w}_{k+1}), \mathcal{L}(\mathbf{w}_{k+1})) \notin \mathcal{F}_{k}$. Augment the filter using \eqref{eq:filterupdate} and continue with the regular iteration in step \ref{alg:itercounterincre} \label{alg:restophase}
\end{enumerate}
\end{algorithmI*}

\begin{remark}\label{rmk:linesearchfiltersimilar}
    Algorithm I differs from Algorithm I in~\cite{wachterglobal} predominantly in steps 3 and 4.2, where the damped Newton step is replaced by the recursive Newton method of section~\ref{ssec:recursivenewton} and forward simulation of section~\ref{ssec:derivforward}. Furthermore, the objective value in the filter criteria is replaced with the Lagrangian and the Armijo condition in step 4.4.1 of Algorithm I in~\cite{wachterglobal} is replaced with condition~\eqref{eq:filterarmijo}.
\end{remark}

\section{Global Convergence}\label{sec:global}

A consequence of Remark~\ref{rmk:linesearchfiltersimilar} is that we only need to show that the recursive Newton and forward simulation steps together behave similarly to a regular Newton step for OCPs of form \eqref{eq:coc} to establish global convergence of the FilterDDP algorithm. 

\subsection{Assumptions} The set $\mathcal{R} \subseteq \mathbb{N}$ is defined as the set of iteration indices in which the feasibility restoration phase is invoked. Denote $\mathcal{R}_\text{inc} \subseteq \mathcal{R}$ the set of iteration counters in which the restoration phase is invoked from step~\ref{alg:restofrombw}. We now state the assumptions necessary for the global convergence analysis of Algorithm 1.

\medskip

\begin{assumptionsg*} Let $\{\mathbf{w}_{k}\}$ be the sequence generated by Algorithm 1, where we assume that the feasibility restoration phase in step 9 always terminates successfully and that the algorithm does not stop in step 2 at a KKT point.

\begin{enumerate}[label=(G{\arabic*})]
    \item There exists an open set $\mathcal{C}$ which contains $(x_{k, t}^+(\alpha), u^+_{k, t}(\alpha))$ for all $k\notin \mathcal{R}_\text{inc}$, $t\in[N]$ and $\alpha\in[0, 1]$ satisfying~\eqref{eq:forwardls}, such that $\ell, f$ and $c$ are differentiable over $\mathcal{C}$ and furthermore, that their function values and derivatives are bounded and Lipschitz continuous over $\mathcal{C}$. \label{ass:difflip}
    \item The Hessian approximations $H_{k,t}$, $B_{k,t}$ and $C_{k,t}$ are uniformly bounded for all $k \notin \mathcal{R}_\text{inc}$ and $t\in[N]$. \label{ass:Hbounded}
    \item The Hessian approximations $H_{k,t}$ are uniformly positive definite on the null space of the Jacobian $A_{k,t}^\top$, i.e., there exists a constant $M_H > 0$ so that for all $k\notin\mathcal{R}_\text{inc}$ and $t\in[N]$,
    \begin{equation}\label{eq:HuuPD}
        \lambda_{\min}\left(Z_{k,t}^\top H_{k, t} Z_{k,t} \right) \geq M_H,
    \end{equation}
    where the columns of $Z_{k,t}\in\mathbb{R}^{n_u \times (n_u - n_c)}$ form an orthogonal basis for the null space of constraint Jacobian $A_{k,t}^\top$. \label{ass:HuuPD}
    \item There exists a constant $M_A > 0$ so that for all $k \notin \mathcal{R}_\text{inc}$ and $t\in[N]$, we have
    \begin{equation}\label{eq:minsingA}
        \sigma_{\min}(A_{k,t}) \geq M_A.
    \end{equation}\label{ass:minsingA}
    \item The iterates for which the restoration phase is invoked from step \ref{alg:restofrombw} (for example because \eqref{eq:HuuPD} or \eqref{eq:minsingA} are violated for any $t\in[N]$) are not arbitrarily close to the feasible region, i.e., there exists a constant $\theta_\text{inc} > 0$ so that $k\notin \mathcal{R}_\text{inc}$ whenever $\theta(\mathbf{w}_{k}) \leq \theta_\text{inc}$. \label{ass:restofeas}
\end{enumerate}
\end{assumptionsg*}

\begin{remark}
    Assumptions~\ref{ass:difflip} and~\ref{ass:Hbounded} establish smoothness and boundedness of the problem data, analogous to Assumtpions (G1) and (G2) of~\cite{wachterglobal}. Assumption~\ref{ass:HuuPD} ensures a certain descent property via Lemma~\ref{lem:descent}, and serves an identical purpose to Lemma 2 in~\cite{wachterglobal}. 
\end{remark}

\begin{remark}
Suppose $Y_{k, t}$ is such that $\left[Y_{k, t} \,\, Z_{k, t}\right]$ form an orthonormal basis of $\mathbb{R}^{n_u}$ and the columns of $Z_{k, t}$ are a basis of the null space of $A_{k, t}^\top$. Then we can decompose $\zeta_{k, t}$ into two orthogonal components
\begin{equation}\label{eq:alphadecomp}
    \zeta_{k, t} = q_{k, t} + p_{k, t},
\end{equation}
where
\begin{equation}\label{eq:nullspaceq}
    q_{k, t} \coloneqq Y_{k, t} \bar{q}_{k, t} \quad \text{and} \quad p_{k, t} \coloneqq Z_{k, t} \bar{p}_{k, t},
\end{equation}
with
\begin{align}
    \bar{q}_{k, t} &\coloneqq - \left(A_{k, t}^\top Y_{k, t} \right)^{-1} c^{k, t}, \label{eq:barqtk} \\
    \bar{p}_{k, t} &\coloneqq -\left( Z_{k, t} ^\top H_{k, t} Z_{k, t} \right)^{-1} Z_{k, t}^\top \left( (Q_u^{k, t})^\top + H_{k, t} q_{k, t} \right). \label{eq:barptk}
\end{align}
\end{remark}

\begin{remark}
    The reduced Hessian given by $Z_{k, t} ^\top H_{k, t} Z_{k, t}$ can be monitored and modified if necessary to ensure~\ref{ass:HuuPD} holds in a practical implementation of the proposed algorithm. Numerical experience~\cite{filterddp} shows that a simple inertia correction procedure which adds a diagonal perturbation to~\eqref{eq:bwlinear} is effective in practice.
\end{remark} 

Similar to~\cite{wachterglobal}, we use  a \emph{first-order criticality measure} $\chi(\mathbf{w}_{k}) \in [0, \infty]$ with the property that if a subsequence $\{\mathbf{w}_{k_i}\}$ of iterates with $\chi(\mathbf{w}_{k_i}) \rightarrow 0$ converges to a feasible limit point $\mathbf{w}^\star$, then $\mathbf{w}^\star$ corresponds to a KKT solution. For the convergence analysis of the proposed algorithm, we define the criticality measure for iterations $k\notin\mathcal{R}_\text{inc}$ as
\begin{equation}\label{eq:criticality}
    \chi(\mathbf{w}_{k}) \coloneqq \sum_{t=1}^N\|\bar{p}_{k, t}\|_2,
\end{equation}
where for completeness, $\chi(\mathbf{w}_{k}) = \infty$ for $k\in\mathcal{R}_\text{inc}$.

To see that $\chi(\mathbf{w}_{k})$ is a criticality measure under Assumptions G, consider a subsequence of iterates $\{\mathbf{w}_{k_i}\}$ with $\lim_{i} \chi(\mathbf{w}_{k_i}) = 0$ and $\lim_i \mathbf{w}_{k_i} = \mathbf{w}^\star$ for some feasible limit point $\mathbf{w}^\star$. Since $\chi(\mathbf{w}_{k_i}) = \infty$ if  $k_i\in\mathcal{R}_\text{inc}$, then we have $k_i \notin \mathcal{R}_\text{inc}$ for $i$ sufficiently large. Furthermore, from Assumption \ref{ass:minsingA} and \eqref{eq:barqtk}, we have that $\lim_i \bar{q}_{k_i, t} = 0$ for all $t\in[N]$ and then from $\lim_{i} \chi(\mathbf{w}_{k_i}) = 0$, \eqref{eq:criticality}, \eqref{eq:barptk} and Assumption \ref{ass:HuuPD}, we have that $\lim_{i\rightarrow \infty} \|Z_{k_i, t}^\top (Q_u^{k_i, t})^\top\| = 0$ for all $t\in[N]$. We will show using the following result that this is an optimality measure for the OCP in~\eqref{eq:coc}.

\begin{lemma}
    Suppose an iterate $k\notin \mathcal{R}_\text{inc}$ is feasible, i.e., $c^{k, t} = 0$ and $x_{k, t+1} = f(x_{k, t}, u_{k, t})$ for $t\in[N-1]$, and furthermore, suppose that $Z_{k, t} ^\top (Q_u^{k, t})^\top = 0$ is satisfied for all $t \in [N]$. Then $\mathbf{x}_{k}, \mathbf{u}_k$ is a KKT point of \eqref{eq:coc}, i.e., there exists $\bm{\phi}^\star \coloneqq \phi^\star_{1:N}$ and $\bm{\lambda}^\star \coloneqq \lambda_{1:N}^\star$ such that~\eqref{eq:kkteq} is satisfied.
\end{lemma}

\begin{proof}
    Since $A_{k, t}$ is a basis for the null space of $Z_{k ,t}^\top$, it follows that there exists some $\phi_{k,t}'$ such that $(Q_u^{k, t})^\top = A_{k, t} \phi_{k, t}' $. Letting
    \begin{equation}\label{eq:optphilamb}
      \phi_{k, t}^\star \coloneqq \phi_{k, t} - \phi_{k, t}', \quad \quad \lambda_{k, t}^\star \coloneqq (V_x^{k, t})^\top,
    \end{equation}
    it follows that 
    \begin{equation}\label{eq:Quzero}
        \setlength\arraycolsep{0pt}
        \begin{array}{ccl}
            0 & = & Q_u^{k, t} - (\phi_{k,t}')^\top A_{k, t}^\top \\
            & \overset{\eqref{eq:gt}}{=} & \ell_u^{k, t} + (\phi_{k, t} - \phi_{k,t}')^\top A_{k, t}^\top + V_x^{k, t+1} f_u^{k, t} \\ 
            & \overset{\eqref{eq:constrkktu}, \eqref{eq:optphilamb}}{=} & \nabla_{u_t} \mathcal{L}(\mathbf{x}_k, \mathbf{u}_k, \bm{\phi}^\star, \bm{\lambda}^\star),
        \end{array}
    \end{equation}
    and so KKT condition~\eqref{eq:constrkktu} is satisfied. 
    
    Furthermore, the value function gradient update becomes
    \begin{equation}\label{eq:Qxzero}
        \setlength\arraycolsep{0pt}
        \begin{array}{ccl}
            0 & \overset{\eqref{eq:vxupdate}}{=} & Q_x^{k, t} + Q_u^{k, t} \beta_{k, t} + (c^{k, t})^\top \omega_{k, t} - V_x^{k, t}  \\
            & = &  Q_x^{k, t} + (\phi_{k, t}')^\top A_{k, t}^\top \beta_{k, t} + (c^{k, t})^\top \omega_{k, t} - V_x^{k, t}\\
            & \overset{\eqref{eq:bwlinear}}{=} &  \ell_x^{k, t} + (\phi_{k, t} - \phi_{k, t}')^\top c_x^{k, t} + V_x^{k, t+1} f_x^{k, t} - V_x^{k, t} \\
            & \overset{\eqref{eq:constrkktx}}{=} & \nabla_{x_t}\mathcal{L}(\mathbf{x}_k, \mathbf{u}_k, \bm{\phi}^\star, \bm{\lambda}^\star),
        \end{array}
    \end{equation}
    and so KKT condition~\eqref{eq:constrkktx} is also satisfied. Finally, the remaining KKT conditions relating to feasibility, i.e.,~\eqref{eq:constrkktdyn} and~\eqref{eq:constrkktconstr}, holds by assumption.
\end{proof}

Our first result shows that the update rule (hence iterates) are bounded, equivalent to Lemma 1 in~\cite{wachterglobal} in the general, nonlinear programming setting.

\begin{lemma}\label{lem:boundedvar}
    Suppose Assumptions G hold. Then there exist constants $M_\zeta$, $M_\beta$, $M_\psi$, $M_\omega$, $M_m > 0$ such that 
    \begin{equation}\label{eq:boundedupdate}
        \begin{gathered}
            \|\zeta_{k, t}\| \leq M_\zeta, \quad \|\beta_{k, t}\| \leq M_\beta, \quad \|\psi_{k, t}\| \leq M_\psi, \\
        \quad \|\omega_{k, t}\| \leq M_\omega , \quad \|V_x^{k, t}\| \leq M_v, \quad |m_k(\alpha)| \leq M_m \alpha, 
        \end{gathered}
    \end{equation}
\end{lemma}
for all $k\notin \mathcal{R}_\text{inc}$, $t\in[N]$ and $\alpha\in(0, 1]$.

\begin{proof}
    We proceed with an inductive argument. Suppose $r_{k, t+1}$ and $P_{k, t+1}$ are uniformly bounded for some $t\in[N]$ and $k\notin \mathcal{R}_\text{inc}$. It follows by Assumptions~\ref{ass:difflip} and~\ref{ass:Hbounded} that the right-hand side of~\eqref{eq:bwlinear} is uniformly bounded. Furthermore, Assumptions~\ref{ass:Hbounded}, \ref{ass:HuuPD} and~\ref{ass:minsingA} ensures that the inverse of the left-hand side in~\eqref{eq:bwlinear} exists and is uniformly bounded. Consequently, $\zeta_{k, t}$, $\psi_{k, t}$, $\beta_{k, t}$ and $\omega_{k, t}$ are uniformly bounded and by~\eqref{eq:vxupdate} and~\ref{ass:difflip}, $r_{k, t}$ and $P_{k, t}$ are uniformly bounded. The base case holds for $t=N$ since $r_{k, N+1} = 0$ and $P_{k, N+1}=0$. It follows that $\zeta_{k, t}$, $\psi_{k, t}$, $\beta_{k, t}$ and $\omega_{k, t}$ are uniformly bounded for all $t\in[N]$. It follows by~\eqref{eq:switching} that $m_k(\alpha) / \alpha $ is also uniformly bounded.
\end{proof}

\begin{remark}
  A direct corollory of Lemma \ref{lem:boundedvar} and Assumption~\ref{ass:difflip} is that the sequence of Lagrange multipliers $\{\bm{\phi}_k\}$ is bounded. As shown in section 4.1 of~\cite{wachterglobal}, this result ensures that the Lagrangian can be used instead of the objective within the filter criteria.
\end{remark}

Next, we establish the descent property equivalent to Lemma 2 in~\cite{wachterglobal} for the general NLP setting, where iterates sufficiently close to feasibility but non-optimal yield a sufficient decrease in the objective function. 

\begin{lemma}\label{lem:descent}
    Suppose Assumptions G hold. If $\{\mathbf{w}_{k_i}\}$ is a subsequence of iterates for which $\chi(\mathbf{w}_{k_i}) \geq \epsilon$ with a constant $\epsilon > 0$ independent of $i$, then there exists constants $\epsilon_1, \epsilon_2 >0$, such that
    \begin{equation}\label{eq:feasimpliesdescent}
        \theta(\mathbf{w}_{k_i}) \leq \epsilon_1 \quad \implies \quad m_{k_i}(\alpha) \leq -\epsilon_2 \alpha
    \end{equation}
    for all $i$ and $\alpha \in (0, 1]$.
\end{lemma}

\begin{proof}
    Consider a subsequence of iterates $\{\mathbf{w}_{k_i}\}$ with $\chi(\mathbf{w}_{k_i}) \geq \epsilon$. By Assumption \ref{ass:restofeas}, for all $\mathbf{w}_{k_i}$ with $\theta(\mathbf{w}_{k_i}) \leq \theta_\text{inc}$, we have $k_i \notin \mathcal{R}_\text{inc}$. Furthermore, $q_{k_i, t} = O(\|c^{k_i,t}\|)$ from \eqref{eq:nullspaceq}, \eqref{eq:barqtk} and \ref{ass:minsingA}. It follows that for all $k_i \notin \mathcal{R}_\text{inc}$,
    \begin{equation}
        \setlength\arraycolsep{0pt}
        \begin{array}{ccl}
          m_{k_i}(\alpha) / \alpha   & = & \sum_{t=1}^N Q_u^{k_i, t} \zeta_{k_i, t}  + \psi_{k_i, t}^\top c^{k_i, t}\\
          &  \overset{\eqref{eq:alphadecomp}, \eqref{eq:nullspaceq}, \eqref{eq:barqtk}}{=} & \sum_{t=1}^N Q_u^{k_i, t} Z_{k_i, t} \bar{p}_{k_i, t}  + Q_u^{k_i, t} Y_{k, t}(A_{k,t}^\top Y_{k, t})^{-1}c^{k, t} \\
          & &+ \psi_{k_i, t}^\top c^{k_i, t} \\
          & \overset{\ref{ass:difflip}, \ref{ass:minsingA}, \eqref{eq:boundedupdate}}{\leq} & \sum_{t=1}^N Q_u^{k_i, t} Z_{k_i, t} \bar{p}_{k_i, t} + c_3 \|c^{k, t} \| \\
          & \overset{\eqref{eq:barptk}}{=} &  \sum_{t=1}^N -\bar{p}_{k_i, t}^\top\left( Z_{k_i, t}^\top H_{k_i, t} Z_{k_i, t} \right) \bar{p}_{k_i, t}  - \bar{p}_{k_i, t}^\top Z_{k_i, t}^\top H_{k_i, t} q_{k_i, t} \\
          & &  + c_3 \|c^{k, t}\| \\
          & \overset{\ref{ass:Hbounded}, \eqref{eq:HuuPD}}{\leq} & \sum_{t=1}^N -c_1 \|\bar{p}_{k_i, t}\|_2^2  +c_2 \|\bar{p}_{k_i, t}\|_2 \|c^{k_i, t}\|  + c_3 \|c^{k, t}\|  \\
          & \leq & \chi(\mathbf{w}_{k_i})\Bigl(-\epsilon \frac{c_1}{N} + c_2\theta(\mathbf{w}_{k_i}) + \dfrac{c_3}{\epsilon} \theta(\mathbf{w}_{k_i})\bigr)
        \end{array}
    \end{equation}
    for some constants $c_1, c_2, c_3 > 0$, where for the last inequality, we used the identity $\left(\sum_{t=1}^N \|\bar{p}_{k_i, t}\|_2 \right)^2 \leq N \sum_{t=1}^N \|\bar{p}_{k_i, t}\|_2^2$ which follows from the Cauchy-Schwarz inequality, $\chi(\mathbf{w}_{k_i}) \geq \epsilon$ and $\sum_{t=1}^N \|\bar{p}_{k_i, t}\|_2 \|c^{k_i, t}\| \leq \chi(\mathbf{w}_{k_i}) \theta(\mathbf{w}_{k_i})$. We now define
    \begin{equation}
        \epsilon_1 \coloneqq \min \left\{ \theta_\text{inc}, \frac{\epsilon^2 c_1}{2N(c_2 \epsilon + c_3)} \right\}.
    \end{equation}
    It follows for all $\mathbf{w}_{k_i}$ with $\theta(\mathbf{w}_{k_i}) \leq \epsilon_1$ that
    \begin{equation}
        m_{k_i}(\alpha) \leq -\alpha \frac{\epsilon c_1}{2N} \chi(\mathbf{w}_{k_i}) \leq -\alpha \frac{\epsilon^2 c_1}{2N}.
    \end{equation}
    The claim follows after defining $\epsilon_2 \coloneqq \frac{\epsilon^2 c_1}{2N}$.
\end{proof}

For our final set of intermediate results in the global convergence analysis, we establish that the update rule and trial point determination (in a sense) linearises the two filter criteria~\eqref{eq:filtermeasures}. This is equivalent to Lemma 3 in~\cite{wachterglobal}.

\begin{lemma}\label{lem:constrlinear}
    Suppose Assumptions G hold. Then there exists $M_\theta > 0$ such that for all $k\notin \mathcal{R}_\text{inc}$ and all $\alpha \in [0, 1]$,
    \begin{equation}\label{eq:constrlinearthm}
        \left| \theta(\mathbf{w}^+_{k}(\alpha)) - (1 - \alpha)\theta(\mathbf{w}_{k}) \right| \leq C_\theta \alpha^2.
    \end{equation}
    
\end{lemma}

\begin{proof}
     First, consider the function $c^+_{k, t}(\alpha) \coloneqq c(x^+_{k, t}(\alpha), u_{k, t}^+(\alpha))$, where $x_{k, t}^+(\alpha)$ and $u_{k, t}^+(\alpha)$ follow \eqref{eq:uforwardls}-\eqref{eq:dynforwardls}. By Assumption \ref{ass:difflipipm} and Taylor's theorem, for all $\alpha \in [0, 1]$,
    \begin{equation}\label{eq:ctaylorls}
      \begin{array}{ccl}
        c^+(\alpha) & = & c(x_{k, t}^+(\alpha), u_{k, t}^+(\alpha)) \\
        & = & c^{k, t} + \alpha \frac{d}{d\alpha} c(x_{k, t}^+(0), u_{k, t}^+(0)) + O(\alpha^2) \\
        & \overset{\eqref{eq:uforwardls}}{=} & c^{k, t} + \alpha \left( c_x^{k, t} \frac{d}{d\alpha} x_{k, t}^+(0) + A_{k, t}^\top \left( \zeta_{k, t} + \beta_{k, t} \frac{d}{d\alpha} x_{k, t}^+(0) \right) \right) + O(\alpha^2) \\
        & \overset{\eqref{eq:bwlinear}}{=} & c^{k, t} + \alpha A_{k, t}^\top \zeta_{k, t} + O(\alpha^2) \\
        & \overset{\eqref{eq:bwlinear}}{=} & (1-\alpha)c^{k, t} + O(\alpha^2),
      \end{array}
    \end{equation}
    where we use \eqref{eq:bwlinear} in line 4 of \eqref{eq:ctaylorls} through relation $c_x^{k, t} = -A_{k, t}^\top \beta_{k, t}$. It follows that 
    \begin{equation}
        \setlength\arraycolsep{0pt}
        \begin{array}{ccl}
            \theta(\mathbf{w}_k^+(\alpha)) & = & \sum_{t=1}^N \|c(x_{k, t}^+(\alpha), u_{k, t}^+(\alpha))\| \\ 
            & \overset{\eqref{eq:ctaylorls}}{=} & (1 - \alpha) \sum_{t=1}^N \|c^{k, t}\| + O(\alpha^2) \\
            & \overset{\eqref{eq:filtermeasures}}{=} & (1- \alpha) \theta(\mathbf{w}_k) + O(\alpha^2),
        \end{array}
    \end{equation}
    which is equivalent to the desired result in \eqref{eq:constrlinearthm}.
\end{proof}

\begin{lemma}\label{lem:valueapprox}
    Suppose Assumptions G hold. Then there exists $C_\mathcal{L} > 0$ such that for all $k\notin \mathcal{R}_\text{inc}$ and all $\alpha \in [0, 1]$,
    \begin{equation}\label{eq:valueapproxthm}
        \left|\mathcal{L}(\mathbf{w}^+_{k}(\alpha)) - \mathcal{L}(\mathbf{w}_{k}) - m_k(\alpha) \right| \leq C_\mathcal{L} \alpha^2.
    \end{equation}
\end{lemma}

\begin{proof} 
    Define function $\mathcal{L}^{k, t}(x, \alpha)$ recursively by
    \begin{equation}\label{eq:lagrangianrecurse}
        \begin{aligned}
            \mathcal{L}^{k,t}(x, \alpha) \coloneqq & \hat{Q}^{k, t}(x, \hat{u}_{k, t}^+(x, \alpha), \hat{\phi}^+_{k, t}(x, \alpha), \alpha),
        \end{aligned}
    \end{equation}
    where $\hat{u}_{k, t}^+$, $\hat{\phi}_{k, t}^+$ are given by
    \begin{subequations}
        \begin{gather}
            \hat{u}_{k, t}^+(x, \alpha) \coloneqq u_{k, t} + \alpha \zeta_{k, t} + \beta_{k, t}(x - x_{k, t}) \\
            \hat{\phi}_{k, t}^+(x, \alpha) \coloneqq \phi_{k, t} + \alpha \psi_{k, t} + \omega_{k, t}(x - x_{k, t}),
        \end{gather}
    \end{subequations}
    base case $\mathcal{L}^{k,N+1}(x, \alpha) \coloneqq 0$ holds and
    \begin{equation}\label{eq:Qhat}
      \hat{Q}^{k, t}(x, u, \phi, \alpha) \coloneqq \ell(x, u) + \phi^\top c(x, u) + \mathcal{L}^{k, t+1}(f(x, u), \alpha).
    \end{equation}
    One can verify by direct substitution that $\mathcal{L}^{k, 1}(\hat{x}_1, \alpha) = L(\mathbf{w}^+_k(\alpha))$ and also, that $\hat{Q}^{k, t}(x, u, \phi, 0) = Q^{k, t}(x, u, \phi)$. It follows that for all $t\in[N]$ that
    \begin{equation}\label{eq:Vxrecursive}
        \setlength\arraycolsep{0pt}
        \begin{array}{ccl}
            \nabla_{x}\mathcal{L}^{k, t}(x_{k, t}, 0) & = & \nabla_{x}\hat{Q}^{k, t}(x_{k, t}, u_{k, t}, \phi_{k, t}, 0) + \nabla_{u}\hat{Q}^{k, t}(x_{k, t}, u_{k, t}, \phi_{k, t}, 0) \beta_{k, t} \\
            & &+ \nabla_{\phi} \hat{Q}^{k, t}(x_{k, t}, u_{k, t}, \phi_{k, t}, 0) \omega_{k, t} \\
            & = & Q^{k, t}_x + Q^{k, t}_u \beta_{k, t} + (c^{k, t})^\top \omega_{k, t} \\
            &\overset{\eqref{eq:vxupdate}}{=} & V_x^{k, t},
        \end{array}
    \end{equation}
    and furthermore,
    \begin{equation}\label{eq:Valpha}
        \begin{aligned}
           & \frac{\partial \mathcal{L}^{k, t}}{\partial \alpha}(x_{k, t}, 0)  \coloneqq Q_u^{k, t} \zeta_{k, t} + \psi_{k, t}^\top c^{k, t} + \frac{\partial \mathcal{L}^{k, t+1}}{\partial \alpha}(x_{k, t+1}, 0).
        \end{aligned}
    \end{equation}
    By applying substitution backwards in time to~\eqref{eq:Valpha}, we conclude that
    \begin{equation}\label{eq:malphaproof}
      \frac{\partial \mathcal{L}^{k, 1}}{\partial \alpha}(\hat{x}_1, 0) = m_k(1).
    \end{equation}
    Finally, since derivatives of $L^{k, 1}(\hat{x}_1, \alpha)$ with respect to $\alpha$ are Lipschitz continuous for all $\alpha \in [0, 1]$ by Assumption~\ref{ass:difflip}, we can apply Taylor's theorem, yielding
    \begin{equation}
        \setlength\arraycolsep{0pt}
        \begin{array}{ccl}
            \mathcal{L}(\mathbf{w}^+_k(\alpha)) & = & \mathcal{L}^{k, 1}(\hat{x}_1, \alpha) \\
            & = & \mathcal{L}^{k, 1}(\hat{x}_1, 0) + \alpha \frac{\partial \mathcal{L}^{k, 1}}{\partial \alpha}(\hat{x}_{1}, 0) + O(\alpha^2) \\
            & \overset{\eqref{eq:malphaproof}}{=} & \mathcal{L}(\mathbf{w}_k) + \alpha m_k(1) + O(\alpha^2) \\
            & \overset{\eqref{eq:switching}}{=} & \mathcal{L}(\mathbf{w}_k) + m_k(\alpha) + O(\alpha^2),
        \end{array}
    \end{equation}
    which is equivalent to the desired result~\eqref{eq:valueapproxthm}.
\end{proof}

The main global convergence result of the filter line search algorithm of~\cite{wachterglobal}, namely Theorem 2, only depends on the damped Newton step through intermediate results Lemmas 1, 2 and 3 of that article. Since we have established the equivalent of these intermediate results for the FilterDDP algorithm through Lemmas~\ref{lem:boundedvar} to~\ref{lem:valueapprox} in this article, the global convergence result of~\cite{wachterglobal} applies to the FilterDDP algorithm.

\begin{theorem}\label{thm:global}
    Suppose Assumptions G hold. Then,
    \begin{equation}
        \lim_{k\rightarrow \infty} \theta(\mathbf{w}_k) = 0
    \end{equation}
    and
    \begin{equation}
        \liminf_{k\rightarrow \infty} \chi(\mathbf{w}_k) = 0.
    \end{equation}
    In other words, all limit points are feasible, and if $\{\mathbf{w}_k\}$ is bounded, then there exists a limit point $\mathbf{w}^\star$ of $\{\mathbf{w}_{k}\}$ which is a first-order optimal point for the OCP \eqref{eq:coc}.
\end{theorem}

\begin{proof}
    See the proof of Theorem 2 in~\cite{wachterglobal}.
\end{proof}

\section{Interior Point Extension for FilterDDP}\label{sec:ipmextension}

We describe an extension to Algorithm I for solving OCPs with additional inequality constraints of form
\begin{equation}\label{eq:cocineq}
    \begin{array}{rl}
    \underset{\mathbf{x}, \mathbf{u}}{\text{minimise}}   & \sum_{t=1}^{N} \ell(x_t, u_t) \\
    \text{subject to} & x_1 = \hat{x}_1 \\
    & x_{t+1} = f(x_t, u_t)  \quad \text{for } t \in [N-1] \\
    & c(x_t, u_t) = 0,  \quad u_t \geq 0 \quad \text{for } t \in [N]
    \end{array}
\end{equation}
and establish its global convergence properties. In particular the extension is to a barrier interior point method of the primal or primal-dual type, similar to section 4.3 of~\cite{wachterglobal} in the general NLP setting. A barrier method solves (approximately, to a fixed tolerance $\epsilon > 0$), a sequence of barrier problems given by a sequence of positive barrier parameters $\{\mu_j\}$, with $\lim_{j\rightarrow \infty} \mu_j = 0$ and takes the form
\begin{equation}\label{eq:barrierproblems}
    \begin{array}{rl}
    \underset{\mathbf{x}, \mathbf{u}}{\text{minimise}}   & \sum_{t=1}^{N} \ell(x_t, u_t) - \mu \sum_{i=1}^m \ln(u_t^{(i)}) \\
    \text{subject to} & x_1 = \hat{x}_1 \\
    & x_{t+1} = f(x_t, u_t)  \quad \text{for } t \in [N-1] \\
    & c(x_t, u_t) = 0  \quad \text{for } t \in [N].
    \end{array}
\end{equation}

A detailed discussion of the convergence of barrier methods as $\mu \rightarrow 0$ is outside of the scope of this article, however, we refer the reader to section 4.3 in~\cite{wachterglobal} for discussion around this point. As a result, we apply our global convergence analysis of the interior point extension of FilterDDP for a fixed barrier parameter $\mu > 0$. To extend the previous sections to the barrier method setting, the objective $\ell$ is replaced with the barrier objective $\varphi_\mu(x_t, u_t) \coloneqq \ell(x_t, u_t) - \mu \sum_{i=1}^m \ln(u_t^{(i)})$ at all instances in the previous sections. There are two important considerations that hold when establishing global convergence of the interior point DDP algorithm, similar to the general, nonlinear programming setting:
\begin{enumerate}
    \item The barrier objective $\varphi_t$ for $t\in[N]$ is only defined if $u_t > 0$ for all $t\in[N]$.
    \item The barrier objective and its derivatives become unbounded as $u_t$ approaches its bound for all $t \in [N]$.
\end{enumerate}
To address the first consideration, the extended algorithm first, initialises $u_{0, t} > 0$ for all $t\in[N]$ and at each iteration, imposes an additional step acceptance criteria where the step size $\alpha_k$ must satisfy a fraction-to-the-boundary condition given by
\begin{equation}\label{eq:fractobound}
    u_{k+1, t } = u_{k, t}^+(\alpha_k) \geq (1 - \tau) u_{k, t} 
\end{equation}
for $\tau \in (0, 1)$, usually chosen close to 1. This property ensures that the sequence of iterates remains within bounds.

\begin{remark}
    By Assumption \ref{ass:difflip}, Lemma \ref{lem:boundedvar} and \eqref{eq:forwardls}, there exists some $\alpha_k$ small enough such that \eqref{eq:fractobound} is satisfied.
\end{remark}

To address the second consideration, we show below in Theorem \ref{thm:boundedawayipm} that the iterates generated by the extended version of Algorithm 1 are bounded away from the bounds. First, we remark that it is necessary to assume that $s_L$ is chosen to be 1 to establish this result. This follows from Remark 6 in~\cite{wachterglobal}, which shows that when $s_L=1$, the barrier objective function and the norm of its gradients do not need to be bounded above (relaxing Assumption \ref{ass:difflip}) to establish global convergence. For later reference, we restate the linear system to be solved for the update rule pertaining to the barrier method:
\begin{equation}\label{eq:bwlinearipm}
    \begin{bmatrix}
        H_{k, t} + \mu U_{k, t}^{-2}  & A_{k, t} \\
        A_{k, t}^\top & 0 
    \end{bmatrix}
    \begin{bmatrix}
        \zeta_{k, t} & \beta_{k, t} \\ \psi_{k, t} & \omega_{k, t}
    \end{bmatrix} = -\begin{bmatrix}
         (Q_u^{k, t} - \mu e^\top U_{k, t}^{-1})^\top & B_{k, t} \\ c^{k, t} & c_x^{k, t}
     \end{bmatrix}
\end{equation}
where $U_{k, t} = \diag(u_{k, t})$. We remark that we can replace the primal barrier term $\mu U_{k, t}^{-2}$ in the top left block of \eqref{eq:bwlinearipm} with a primal-dual equivalent $\Sigma_{k, t} \coloneqq U_{k, t}^{-1}Z_{k, t}$, with dual variables $z_{k, t} > 0$, as long as there exists $m_\Sigma > 1$ such that
\begin{equation}
    \frac{1}{m_\Sigma}\mu \leq u_{k, t}^{(i)} z_{k, t}^{(i)} \leq m_\Sigma \mu,
\end{equation}
for all $i\in[n_u]$, $t \in [N]$ and $k$. We now state the assumptions required to establish convergence of the barrier interior point algorithm:

\begin{assumptionsb*} Let $\{\mathbf{w}_{k}\}$ be the sequence generated by Algorithm 1 adapted to the barrier problem, from starting point $\mathbf{w}_0$, where we assume that the feasibility restoration phase in step 9 always terminates successfully with $\mathbf{w}$ for which $\mathbf{u} > 0$, and that the algorithm does not stop in step 2 at a KKT point.
    \begin{enumerate}[label=(B{\arabic*})]
        \item There exists an open set $\mathcal{C}$ which contains $(x_{k, t}^+(\alpha), u^+_{k, t}(\alpha))$ for all $k\notin \mathcal{R}_\text{inc}$, $t\in[N]$ and $\alpha\in[0, 1]$, such that $\ell, f$ and $c$ are differentiable over $\mathcal{C}$ and furthermore, that their function values and derivatives are bounded and Lipschitz continuous over $\mathcal{C}$. \label{ass:difflipipm}
        \item The Hessian approximations $H_{k,t}$, $B_{k,t}$ and $C_{k,t}$ are uniformly bounded for all $k \notin \mathcal{R}_\text{inc}$ and $t\in[N]$. \label{ass:Hboundedipm}
        \item The Hessian approximations $H_{k,t} + \mu U_{k, t}^{-2}$ are uniformly positive definite on the null space of the Jacobian $A_{k,t}^\top$, i.e., there exists a constant $M_H > 0$ so that for all $k\notin\mathcal{R}_\text{inc}$ and $t\in[N]$,
        \begin{equation}\label{eq:HuuPDipm}
            \lambda_{\min}\left(Z_{k,t}^\top \left(H_{k, t} + \mu U_{k, t}^{-2} \right) Z_{k,t} \right) \geq M_H,
        \end{equation}
        where the columns of $Z_{k,t}\in\mathbb{R}^{n_u \times (n_u - n_c)}$ form an orthogonal basis for the null space of constraint Jacobian $A_{k,t}^\top$. \label{ass:HuuPDipm}
        \item There exists a constant $M_A > 0$ so that for all $k \notin \mathcal{R}_\text{inc}$ and $t\in[N]$, we have
        \begin{equation}\label{eq:minsingAipm}
            \sigma_{\min}(A_{k,t}) \geq M_A.
        \end{equation}\label{ass:minsingAipm}
        \item The iterates for which the restoration phase is invoked from step \ref{alg:restofrombw} (for example because \eqref{eq:HuuPDipm} or \eqref{eq:minsingAipm} are violated for any $t\in[N]$) are not arbitrarily close to the feasible region, i.e., there exists a constant $\theta_\text{inc} > 0$ so that $k\notin \mathcal{R}_\text{inc}$ whenever $\theta(\mathbf{w}_{k}) \leq \theta_\text{inc}$. \label{ass:restofeasipm}
        \item The sequence of iterates $\{\mathbf{w}_{k}\}$ are uniformly bounded. \label{ass:boundediteratesipm}
        \item At all feasible limit points $\bar{\mathbf{u}}$ of $\{\mathbf{u}_k\}$, with corresponding state trajectory $\bar{\mathbf{x}}$, the gradients of the active constraints, 
        \begin{equation}
            \nabla_{u_t} c(\bar{x}_t, \bar{u}_t)^{(1)}, \dots, \nabla_{u_t} c(\bar{x}_t, \bar{u}_t)^{(n_c)}, \, \text{ and } \,  e_i \,  \text{ for } \,  i \in \{j : \bar{u}_t^{(j)} = 0\},
        \end{equation}
        are linearly independent, where $e_i$ denotes a vector with a 1 in element $i$ and zeros elsewhere. \label{ass:independentgrad}
        \item There exist constants $\tilde{\delta}_\theta, \tilde{\delta}_u > 0$ so that whenever the restoration phase is called in step 9 in an iteration $k \in \mathcal{R}$ with $\theta(\mathbf{w}_k) \leq \tilde{\delta}_\theta$, it returns a new iterate $\mathbf{w}_{k+1}$, where $u_{k+1, t}^{(i)} \geq u_{k, t}^{(i)}$ for all $t, i$ satisfying $u_{k, t}^{(i)} \leq \tilde{\delta}_u$.
    \end{enumerate}
\end{assumptionsb*}

\begin{remark}
    Assumptions B adapt Assumptions G for the barrier interior point problem in an almost identical way to the extension presented in~\cite{wachterglobal} section 4.3 for the general NLP setting. A detailed discussion around the reasonability of the assumptions is provided in section 4.3 of~\cite{wachterglobal}. 
\end{remark}

The remainder of this section is focused on establishing the following result, which is equivalent to Theorem 3 in~\cite{wachterglobal}.

\begin{theorem}\label{thm:boundedawayipm}
    Suppose Assumptions B hold. Then there exists a constant $\epsilon_u > 0$ so that $\mathbf{u}_k \geq \epsilon_u e$ for all $k$.
\end{theorem}

A result of Theorem \ref{thm:boundedawayipm} is that the barrier objective and its derivatives are bounded across the sequence of iterates generated by the extended algorithm, implying Assumptions G hold for the barrier problem \eqref{eq:barrierproblems} and thus, the global convergence result of Theorem \ref{thm:global} applies to the interior point extension described in this section. We further remark there exists some $\alpha_k^{\max} \in (0, 1]$ bounded away from zero which satisfies \eqref{eq:fractobound}, which is necessary for the global convergence result in Theorem \ref{thm:global} (see~\cite{wachterglobal} for more details). To show this, we apply Taylor's theorem to \eqref{eq:uforwardls}, yielding
\begin{equation}
    u_{k, t}^+(\alpha) = u_{k, t} + \alpha \frac{d}{d\alpha}u_{k, t}^+(0) + O(\alpha^2),
\end{equation}
noting that $\frac{d}{d\alpha}u_{k, t}^+(0)$ is uniformly bounded for all $k$ and $t$ by Assumption \ref{ass:difflipipm} and Lemma \ref{lem:boundedvar}. The result then follows from \eqref{eq:fractobound}. Before we prove Theorem \ref{thm:boundedawayipm}, we make use of the following intermediate result which is equivalent to Lemma 11 in~\cite{wachterglobal}, adapted for the FilterDDP algorithm.

\begin{lemma}\label{lem:ipm}
    Suppose Assumptions B hold. Then for a given subset $\mathcal{S} \subseteq [N]\times [n_u]$ and a constant $\delta_l > 0$, there exist $\delta_s, \delta_\theta > 0$ so that $u_{k+1, t}^{(i)} - u_{k, t}^{(i)} \geq 0$ for $(t, i) \in \mathcal{S}$ whenever $k \notin \mathcal{R}$ and
    \begin{equation}\label{eq:impG}
        \begin{gathered}
            \mathbf{w}_k  \in G \coloneqq \biggl\{ \mathbf{w} : \mathbf{u} \geq 0, x_{t+1} = f(x_t, u_t) \, \text{ for } \, t \in [N-1],  \\ u_t^{(i)} \leq \delta_s \text{ for } (t, i)\in\mathcal{S}, u^{(i)}_t \geq \delta_l 
        \text{ for } (t, i) \notin \mathcal{S}, \theta(\mathbf{w}) \leq \delta_\theta \biggr\};
        \end{gathered}
    \end{equation}
    i.e., at sufficiently feasible points, the update rule moves the iterate away from almost active bounds.
\end{lemma}

\begin{proof}
    The proof is a minor adaptation of the proof of Lemma 11 in \cite{wachterglobal}. Let us denote with $u_{k, t}^s$ the components of $u_{k, t}$ in $\mathcal{S}$ and $u_{k, t}^l$ the remaining ones. Without loss of generality, we assume $u_{k, t} = (u_{k, t}^s, u_{k, t}^l)$ for all $t\in[N]$; Similarly, define $A_{k, t}^s$, $A_{k, t}^l$ etc. First, rewrite the linear system \eqref{eq:bwlinearipm} by scaling the first rows and columns by $U_{k, t}^s \coloneqq \mathrm{diag}(u_{k, t}^s)$:
    \begin{equation}\label{eq:bwlinearipmscaled}
        \begin{gathered}
        \begin{bmatrix}
            U_{k, t}^s H_{k, t}^{ss} U_{k, t}^s + \mu I  & U_{k, t}^s H_{k, t}^{sl} & U_{k, t}^s A_{k, t}^s \\
            H_{k, t}^{ls} U_{k, t}^s &  H_{k, t}^{ll} + \mu (U_{k, t}^l)^{-2} & A_{k, t}^l \\
            (A_{k, t}^s)^\top U_{k, t}^s & (A_{k, t}^l)^\top & 0 
        \end{bmatrix} \\
        \begin{bmatrix}
            \tilde{\zeta}_{k, t}^s & \tilde{\beta}_{k, t}^s \\ \zeta_{k, t}^l & \beta_{k, t}^l \\ \psi_{k, t} & \omega_{k, t}
        \end{bmatrix} 
          = -\begin{bmatrix}
             ((Q_u^{k, t})^s U_{k, t}^s - \mu e^\top)^\top & U_{k, t}^s B_{k, t}^s \\ ((Q_u^{k, t})^l - \mu e^\top (U_{k, t}^l)^{-1})^\top & B_{k, t}^l \\c^{k, t} & c_x^{k, t}
         \end{bmatrix},
        \end{gathered}
    \end{equation}
    where we define $\tilde{\zeta}_{k, t}^s \coloneqq (U_{k, t}^s)^{-1}\zeta_{k, t}^s $ and $\tilde{\beta}_{k, t}^s \coloneqq (U_{k, t}^s)^{-1}\beta_{k, t}^s $. For some initial choice of $\delta_s, \delta_\theta > 0$, let $\bar{\mathbf{u}}$ be a feasible point with $\bar{u}_{t}^s = 0$ for all $t\in[N]$. We have from Assumption \ref{ass:independentgrad} that the columns of the matrix
    \begin{equation}
        \begin{bmatrix}
            \nabla_{u_t} c(\bar{x}_t, \bar{u}_t)^s & I \\
            \nabla_{u_t} c(\bar{x}_t, \bar{u}_t)^l & 0
        \end{bmatrix},
    \end{equation}
    and therefore the columns of $\nabla_{u_t}c(\bar{x}_t, \bar{u}_t)^l$ are linearly independent for all $t\in[N]$. Using a compactness argument and Assumption \ref{ass:boundediteratesipm}, we can find a constant $m_\sigma > 0$ so that $\sigma_{\min}(\nabla_{u_t}c(\bar{x}_t, \bar{u}_t)^l) \geq m_\sigma$ for all feasible limit points $\bar{\mathbf{u}}$ of $\{\mathbf{u}_k\}$ with corresponding state trajectory $\bar{\mathbf{x}}$ when $\bar{u}^s_t = 0$. Therefore, we have from Assumption \ref{ass:difflipipm} that $\sigma_{\min}(A_{k, t}^l) \geq m_\sigma / 2$ for all $\mathbf{w}_k \in G$ if $\delta_\theta$ and $\delta_s$ are chosen sufficiently small.

    In addition, possibly after further decreasing $\delta_\theta$, it follows from Assumptions \ref{ass:HuuPDipm} and \ref{ass:restofeasipm} that for all $\mathbf{w}_k \in G$, the projection of $H^{ll}_{k, t} + \mu (U_{k, t}^l)^{-2}$ into the null space of $(A_{k, t}^l)^\top$ is uniformly positive definite. Together with the boundedness assumptions \ref{ass:difflipipm} and \ref{ass:Hboundedipm}, we see that \eqref{eq:bwlinearipmscaled} satisfies
    \begin{equation}\label{eq:bwlinearipmscaled1}
        \begin{gathered}
        \left(
        \begin{bmatrix}
        \mu I & 0 & 0 \\
        0 & H_{k, t}^{ll} + \mu (U_{k, t}^l)^{-2} & A_{k, t}^l \\
        0 & (A_{k, t}^l)^\top & 0
        \end{bmatrix}
        + O(\delta_s) \right) \begin{bmatrix}
            \tilde{\zeta}_{k, t}^s & \tilde{\beta}_{k, t}^s \\ \zeta_{k, t}^l & \beta_{k, t}^l \\ \psi_{k, t} & \omega_{k, t}
        \end{bmatrix} \\
         = - \begin{bmatrix}
            -\mu e & 0 \\
            ((Q_u^{k, t})^l - \mu e^\top (U_{k, t}^l)^{-1})^\top & B_{k, t}^l \\
            c^{k, t} & c_x^{k, t}
        \end{bmatrix} + O(\delta_s),
        \end{gathered}
    \end{equation}
    for $\mathbf{w}_k \in G$. Note, the inverse of the matrix in the square brackets in the left-hand side of \eqref{eq:bwlinearipmscaled1}, as well as the right-hand side are uniformly bounded for $\delta_s$ sufficiently small. Therefore, for $\mathbf{w}_k \in G$, we have that
    $\tilde{\zeta}_{k, t}^s = e + O(\delta_s)$ and $\tilde{\beta}_{k, t}^s = O(\delta_s)$. Finally, multiplying by $U_{k, t}^s$ yields
    \begin{equation}\label{eq:alphabetaipms}
        \zeta_{k, t}^s = u_{k, t}^s + O(\delta_s^2) \quad \text{and} \quad \beta_{k, t}^s = O(\delta_s^2),
    \end{equation}
    and therefore the desired result follows by
    \begin{equation}\label{eq:movesup}
        \begin{array}{ccl}
            u_{k+1, t}^s - u_{k, t}^s & = & \alpha_k \zeta_{k, t}^s + \beta_{k, t}^s(x_{k+1, t} - x_{k, t}) \\
            & \overset{\eqref{eq:alphabetaipms}}{=} & \alpha_k (u_{k, t}^s + O(\delta_s^2)) + O(\delta_s^2) \\
            & \geq &  0,
        \end{array}
    \end{equation}
    for $\delta_s$ sufficiently small.
\end{proof}

To conclude, we remark that the proof of Theorem \ref{thm:boundedawayipm} follows exactly as the proof of Theorem 3 in~\cite{wachterglobal} after replacing Lemma 11 in that paper with Lemma \ref{lem:ipm}. 

\section{Conclusion}

In this article, we established the global convergence of FilterDDP, a line-search filter differential dynamic programming algorithm for solving discrete-time, constrained optimal control problems. The global convergence of the algorithm is established by observing that the DDP backward recursive Newton and forward simulation phases together yield identical properties to a Newton step on the KKT conditions in the context of a global convergence analysis. As a result, the global convergence analysis for the general, nonlinear programming setting~\cite{wachterglobal} can be used directly to establish the global convergence of the FilterDDP algorithm.

An interesting direction for future work is to establish the fast local convergence of FilterDDP without ignoring the step acceptance criteria as in Xu et al.~\cite{filterddp}. This property was established in the general NLP setting by W\"{a}chter and Biegler~\cite{wachterlocal} by including a second-order corrections step. We envision that an analogous second-order corrections step will admit a similar convergence property for the FilterDDP algorithm. Another promising direction for future work involves proposing and establishing the global convergence of a sequential quadratic programming variant of FilterDDP for the general, inequality constrained setting.

\section*{Acknowledgments}
This work was supported by the Australian Research Council under grant DP250101763 and the United States Air Force Office of Scientific Research
under Grant No. FA2386-24-1-4014.

\bibliographystyle{unsrt}
\bibliography{references}
\end{document}